\documentclass{article}
\usepackage{amsmath,amssymb,latexsym,amsthm}

\newcommand{\PP}{\mathbb{P}}
\newcommand{\ZZ}{\mathbb{Z}}
\newcommand{\Ptre}{\PP^3}
\newcommand{\Pfour}{\PP^4}
\newcommand{\Pn}{\PP^n}
\newcommand{\OO}{\mathcal{O}}
\newcommand{\OPuno}{\OO_{\PP^1}}
\newcommand{\Odue}{\OO_{Q_2}}
\newcommand{\Otre}{\OO_{Q_3}}
\newcommand{\On}{\OO_{Q_n}}
\newcommand{\OD}{\OO_D}
\newcommand{\EH}{E\vert_H}
\newcommand{\ED}{E\vert_D}
\newcommand{\OH}{\mathcal{O}_{H}}
\newcommand{\shI}{\mathcal{I}}
\newcommand{\Ext}{\mathrm{Ext}}
\newcommand{\mm}{\mathfrak{m}}
\newcommand{\CH}{\mathrm{CH}}
\newcommand{\N}{\mathcal{N}}
\newcommand{\F}{\mathcal{F}}
\newcommand{\Ss}{\mathcal{S}}
\newcommand{\coker}{\mathrm{coker}\kern.5pt}

\theoremstyle{plain}
\newtheorem{theorem}{Theorem}[section]
\newtheorem{corollary}[theorem]{Corollary}
\newtheorem{proposition}[theorem]{Proposition}
\newtheorem{lemma}[theorem]{Lemma}
\newtheorem{thm}{Theorem}
\newtheorem*{theo}{Theorem}

\theoremstyle{definition}
\newtheorem{definition}[theorem]{Definition}

\newtheorem{remark}[theorem]{Remark}

\begin{document}

\title{On Buchsbaum bundles on quadric hypersurfaces}

\author{E. Ballico \and F. Malaspina \and P. Valabrega \and M. Valenzano
\thanks{The paper was written while all authors were members of INdAM-GNSAGA. \hfill\break Lavoro eseguito con il supporto del progetto PRIN \lq\lq Geometria delle variet\`a algebriche e dei loro spazi di moduli\rq\rq, cofinanziato dal MIUR (cofin 2008).}}

\date{}

\maketitle

\begin{abstract}
Let $E$ be an indecomposable rank two vector bundle on the projective space $\PP^n, n \ge 3$, over an algebraically closed field of characteristic zero.  It is well known that $E$ is arithmetically Buchsbaum if and only if  $n=3$ and $E$ is a null-correlation bundle. 
In the present paper we establish an analogous result for rank two indecomposable arithmetically Buchsbaum vector bundles on the smooth quadric hypersurface $Q_n\subset\PP^{n+1}$, $n\ge 3$. We give in fact a full classification and prove that   $n$ must be at most $5$. 
As to $k$-Buchsbaum rank two vector bundles on $Q_3$, $k\ge2$, we prove two boundedness results.
\\[5pt]
\\[5pt]
\textbf{MSC\,2010:} 14F05.
\end{abstract}

\section{Introduction}

Many papers have been written on $k$-Buchsbaum indecomposable rank two vector bundles on projective $n$-spaces, $n\ge3$, (see for instance \cite{Chang}, \cite{es}, \cite{EF}, \cite{els},  \cite{mkr}). 
\\
We recall that a vector bundle is called \emph{arithmetically Cohen-Macaulay} (i.e.\ 0-Buchsbaum) if it has no intermediate cohomology, while it is called \emph{arithmetically Buchsbaum} (i.e.\ 1-Buchsbaum) if it has all the intermediate cohomology modules with trivial structure. Moreover, we say that a bundle is \emph{properly} arithmetically Buchsbaum if it is so, but it is not arithmetically Cohen-Macaulay.
The following  result about rank two vector bundles on a projective space is well-known (see \cite{Chang} and \cite{EF}):

\begin{theo}
Let $E$ be an arithmetically Buchsbaum, normalized, rank 2 vector bundle on the projective space $\Pn$, $n\ge3$. Then $E$ is one of the following:
\begin{enumerate}
\item $n\ge3\colon$ $E$ is a split bundle;
\item $n=3\colon$ $E$ is stable with $c_1=0$, $c_2=1$, i.e.\ $E$ is a null-correlation bundle.
\end{enumerate}
\end{theo}

\noindent
Therefore, on projective spaces the only properly arithmetically Buchsbaum rank two vector bundles are the null-correlation bundles on $\Ptre$.

In the present paper we investigate arithmetically Buchsbaum rank two vector bundles on a smooth quadric hypersurface $Q_n\subset\PP^{n+1}$ and can give a full classification, getting the following:

\begin{thm}\label{main}
Let $E$ be an arithmetically Buchsbaum, normalized, rank 2 vector bundle on a smooth quadric hypersurface $Q_n$, $n\ge3$. Then $E$ is one of the following:
\\[5pt]
i) arithmetically Cohen-Macaulay bundles:
\begin{enumerate}
\item $n\ge3\colon$ $E$ is a split bundle;
\item $n=3\colon$ $E$ is stable with $c_1=-1$, $c_2=1$, i.e.\ $E$ is a spinor bundle;
\item $n=4\colon$ $E$ is stable with $c_1=-1$, $c_2=(1,0)$ or $(0,1)$, i.e.\ $E$ is a spinor bundle;
\end{enumerate}
ii) properly arithmetically Buchsbaum bundles:
\begin{enumerate}
\item[4.] $n=3\colon$ $E$ is stable with $c_1=-1$, $c_2=2$, i.e.\ $E$ is associated to two skew lines or to a double line;
\item[5.] $n=3\colon$ $E$ is stable with $c_1=-1$, $c_2=3$, and $H^0(Q_3,E(1))=0$, i.e.\ $E$ is associated to a smooth elliptic curve of degree $7$ in $Q_3\subset\Pfour$;
\item[6.] $n=4\colon$ $E$ is stable with $c_1=-1$, $c_2=(1,1)$, i.e.\ $E$ is the restriction of a Cayley bundle to $Q_4$;
\item[7.] $n=5\colon$ $E$ is stable with $c_1=-1$, $c_2=2$, i.e.\ $E$ is a Cayley bundle;
\item[8.] $n\ge6\colon$ no properly arithmetically Buchsbaum bundle exists.
\end{enumerate}
\end{thm}

As for $k$-Buchsbaum bundles, $k\ge2$, on a quadric threefold $Q_3$, we prove two boundedness results for the second Chern class $c_2$, both in the stable and in the non-stable case.

\section{General facts and notation}

We work over an algebraically closed field $k$ of characteristic zero.

\subsection{Quadric hypersurfaces $Q_n \subset \PP^{n+1}$}

We denote by $Q_n$ any smooth quadric hypersurface in the projective space $\PP^{n+1}$, with $n\ge2$.
We recall some well known facts about quadric hypersurfaces. 
Firstly, for the first and second Chow groups of $Q_n$ there are natural isomorphisms
$$\CH^i(Q_n) \cong \ZZ \quad i=1,2$$
with the following exceptions:
$$\CH^1(Q_2) \cong \ZZ\oplus\ZZ\quad\text{and}\quad\CH^2(Q_5)\cong\ZZ\oplus\ZZ,$$
so an element of these particular Chow groups is identified with a pair $(a,b)$ of integers.
\\
Now, let $Q_3\subset\PP^4$ be a smooth quadric threefold. Take a general hyperplane section $Q_2$ and a general conic section $D$ of $Q_3$, i.e.\ 
$$Q_2 = Q_3 \cap K, \qquad D = Q_3 \cap K \cap K' = Q_2 \cap K',$$
where $K$ and $K'$ are two general hyperplanes in $\PP^4$.
We have the following two inclusion maps $D \lhook\!\xrightarrow{\;j\;} Q_2 \lhook\!\xrightarrow{\;i\;} Q_3$, so it holds, for the \lq\lq hyperplane section\rq\rq\ line bundles,
$$i^* \Otre(1) = \Odue(1,1)$$
and
$$j^* \Odue(1,1) = \OD(1) \simeq \OPuno(2),$$
since the irreducible conic $D$ is the isomorphic image through the 2-fold Veronese embedding of the projective line $\PP^1$.  
The above notation $\Odue(1,1)$ has the usual meaning
$$\Odue(1,1) \simeq \pi_1^*\OPuno(1) \otimes \pi_2^*\OPuno(1),$$
where $\pi_1$ and $\pi_2$ are the projections onto the two factors in the standard isomorphism
$Q_2 \cong \PP^1 \times \PP^1,$ 
and moreover, for every integer $t$,
$$\Odue(t,t) = \Odue(1,1)^{\otimes t} \simeq \pi_1^*\OPuno(t) \otimes \pi_2^*\OPuno(t).$$

\subsection{Rank 2 vector bundles on $Q_n$}

Let $E$ be a rank 2 vector bundle on a smooth quadric hypersurface $Q_n \subset \PP^{n+1}$ with $n\ge 3$.

\begin{definition}
The bundle $E$ is called \emph{normalized} if it has first Chern class $c_1\in\{0,-1\}$.
We define the \emph{first relevant level} of $E$ as the integer 
$$\alpha = \alpha(E) := \min\{t\in\ZZ \mid h^0(Q_n,E(t)) \ne 0\}.$$
Let $H$ be a general hyperplane section of $Q_n$. We set $\EH:=E\otimes \OH$, that is $\EH$ is the restriction of the bundle $E$ to the smooth subvariety $H \cong Q_{n-1}$.
So we define the first relevant level of the restricted bundle $\EH$ as the whole number
$$a = a(E) := \alpha(\EH) = \min\{t\in\ZZ \mid h^0(Q_{n-1},\EH(t)) \ne 0 \},$$
with the convention that, when $n=3$, $a$ is the first relevant level of $\EH$ with respect to the line bundle $\Odue(1,1)$.
\\
It is easy to see that for every vector bundle $E$ it holds: $a \le \alpha$.
\end{definition}

\begin{definition}
Let $E$ be a rank 2 vector bundle on $Q_n$ with first Chern class $c_1$ and first relevant level  $\alpha$. 
We say that $E$ is \emph{stable} if $2\alpha+c_1 >0$, or equivalently, if $\alpha >0$ when $E$ is normalized.
We say that $E$ is \emph{semistable} if $2\alpha+c_1 \ge0$, or equivalently, if $\alpha \ge -c_1$ when $E$ is normalized. 
Obviously every stable bundle is semistable. Conversely the only semistable bundles which are not stable are those with $c_1=\alpha=0$.
\\
We say that $E$ is \emph{non-stable} if $2\alpha+c_1\le 0$, that is if $\alpha\le 0$ when $E$ is normalized.
\end{definition}

\begin{definition}
We say that $E$ is an \emph{extendable bundle} or that it extends to a bundle on $Q_{n+1}$ if there exists a rank $2$ vector bundle $F$ on $Q_{n+1}$ such that $E=F\vert_{Q_n}$, where $Q_n\subset Q_{n+1}$ is a general hyperplane section of $Q_{n+1}$.
\end{definition}

\begin{definition}
We say that $E$ is a \emph{split bundle} if it is (isomorphic to) the direct sum of two line bundles, i.e.\ $E=\On(a) \oplus \On(b)$ for suitable integers $a$ and $b$. Obviously each split bundle is non-stable.
\end{definition}

\begin{definition}
Let $E$ be rank $2$ vector bundle on a smooth quadric $Q_n \subset \PP^{n+1}$, with $n\ge3$. We set 
$$R = \mathop\oplus_{t\ge0} H^0(Q_n,\On(t))\quad\text{and}\quad \mm = \mathop\oplus_{t>0} H^0(Q_n,\On(t)),$$
and also
$$H^i_*(Q_n,E) = \mathop\oplus_{t\in\ZZ} H^i(Q_n,E(t))\quad\text{for }i=0,\dots,n,$$
which are modules of finite length on the ring $R$. \\
We say that $E$ is \emph{$k$-Buchsbaum}, with $k\ge0$, if for all integers $p$, $q$ such that $1\le p\le q-1$ and $3\le q\le n$ it holds
$$\mm^k \cdot H^p_*(Q',E\vert_{Q'}) = 0,$$
where $Q'$ is a general $q$-dimensional linear section of $Q_n$, i.e.\ $Q'$ is a quadric hypersurface cut out on $Q_n$ by a general linear space $L\subset\PP^{n+1}$ of dimension $q+1$, that is $Q'=Q_n \cap L$.
\\
Obviously $E$ is $k$-Buchsbaum if and only if $E(t)$ is $k$-Buchsbaum for every $t\in\ZZ$. Moreover, if $E$ is $k$-Buchsbaum, then $E$ is $k'$-Buchsbaum for all $k' \ge k$.
So we say that $E$ is \emph{properly $k$-Buchsbaum} if it is $k$-Buchsbaum but not $(k-1)$-Buchsbaum.
\\
Notice that $E$ is $0$-Buchsbaum if and only if $E$ has no intermediate cohomology, i.e.\  $H^i(Q_n,E(t))=0$ for every $t\in\ZZ$ and $1\le i\le n-1$. Such a bundle is also called \emph{arithmetically Cohen-Macaulay}.
\\
Observe also that $E$ is $1$-Buchsbaum if and only if $E$ has every intermediate cohomology module with trivial structure over $R$. Such a bundle is also called  \emph{arithmetically Buchsbaum}.
\end{definition}

Now we recall some known facts.

\begin{theorem}[Castelnuovo-Mumford criterion]\label{CM}
Let $F$ be a coherent sheaf on $Q_n$ such that $H^i(Q_n,F(-i))=0$ for $i>0$. Then $F$ is generated by global sections and $H^i(Q_n,F(-i+j))=0$ for $i>0$, $j\ge0$.
\end{theorem}
\begin{proof}
See \cite{HeS}.
\end{proof}

\begin{theorem}\label{acm}
Let $F$ be a vector bundle on $Q_n$, $n\ge3$. If $F$ is arithmetically Cohen-Macaulay, i.e.\ it has no intermediate cohomology, then $F$ is a direct sum of line bundles and twisted spinor bundles.
\end{theorem}
\begin{proof}
See \cite{Sols}.
\end{proof}

\begin{remark}
For an account of spinor bundles on quadrics see \cite{Ot2} and \cite{Sols}.
\end{remark}

\begin{definition}\label{Cb1}
As introduced in \cite{Ott}, a \emph{Cayley bundle} $C$ on a smooth quadric hypersurface  $Q_5$ is a bundle arising from the following two exact sequences
$$0 \to \On \to S^\ast \to G \to 0$$
$$0 \to \On \to G^\ast(1) \to C(1) \to 0$$
where $S$ is the spinor bundle on $Q_5$ and $G$ is a stable rank 3 vector bundle with Chern classes $c_1=c_2=c_3=2$ defined by the generic section of $S^\ast$, and moreover $C$ is defined by a nowhere vanishing section of $G^\ast(1)$.
\end{definition}

\begin{theorem}\label{Cb2}
Each stable rank 2 vector bundle on $Q_5$ with Chern classes $c_1=-1$ and $c_2=1$ is a Cayley bundle, and these bundles do not extend to $Q_6$. Moreover, the zero locus of a general global section of $C(2)$, $C$ a Cayley bundle, is (isomorphic to) the complete flag threefold $F(0,1,2)$ of linear elements of $\PP^2$.
\end{theorem}
\begin{proof}
See \cite{Ott}.
\end{proof}

\begin{theorem}\label{Cb3}
Each stable rank 2 vector bundle $F$ on $Q_4$ with Chern classes $c_1=-1$ and $c_2=(1,1)$ extends in a unique way to a Cayley bundle on $Q_5$. Moreover, such a bundle has a global section of $F(1)$ vanishing on two disjoint 2-planes.
\end{theorem}
\begin{proof}
See \cite{Ott}.
\end{proof}

\subsection{Rank 2 vector bundles on $Q_3$}

Let $E$ be a rank 2 vector bundle on a smooth quadric threefold $Q_3$. As usual we will use the notation $E(t)$ for the twisted bundle $E\otimes\Otre(t)$, for all $t\in\ZZ$.
\\
We identify the Chern classes $c_1$ and $c_2$ of $E$ with integers, and we recall the well known following formulas:
\begin{align*}
& c_1(E(t)) = c_1 + 2 t \\
& c_2(E(t)) = c_2 + 2 c_1 t + 2 t^2.
\end{align*}
Moreover, the Hilbert polynomial of the vector bundle $E$ is
\begin{equation}\label{chiE}
\chi(E(t)) = \begin{cases}  (2t +3)(2t^{2} + 6t + 4 - 3c_2)/6 & \text{if}\ \ c_1=0 \\[5pt]
(t +1)(2t^{2} + 4t + 3 - 3c_2)/3 & \text{if}\ \ c_1=-1
\end{cases}
\end{equation}
(see e.g.\ \cite{val}, Proposition~9 and Corollary~2).
In particular we have
$$\chi(E(-1)) = {} - \frac{c_2}{2}, \quad \chi(E) = 2 - \frac{3}{2} c_2, \quad \chi(E(1)) = 10 - \frac{5}{2} c_2,$$
when $c_1=0$, while
$$\chi(E(-1)) = 0, \quad \chi(E) = 1 - c_2, \quad \chi(E(1)) = 6 - 2 c_2.$$
when $c_1=-1$.
\\
If $E$ is a rank 2 vector bundle on $Q_3$, then we define
$$\EH := E\otimes\OH$$
the restriction of the bundle $E$ to a general hyperplane section $H\cong Q_2$ (zero locus of a general global section of $\Otre(1)$). Moreover we will use the following convention:
$$\EH(t) := \EH \otimes \Odue(t,t)\quad\forall\,t\in\ZZ.$$
Furthermore, we define the restriction $\ED$ of the bundle $E$, and hence of the bundle $\EH$, to a general conic section $D$ of the threefold $Q_3$ as:
$$\ED := E \otimes \OD = \EH \otimes \OD.$$
So we have, for each integer $t$, the following exact sequences
\begin{equation}\label{rH}
0 \to E(t-1) \to E(t) \to \EH(t) \to 0
\end{equation}
and
\begin{equation}\label{rD}
0 \to \EH(t-1) \to \EH(t) \to \ED(t) \to 0,
\end{equation}
the so called \emph{restriction sequences}.
\\
Recall that, by our convention, sequence (\ref{rH}) is in fact
$$0 \to E(t-1) \to E(t) \to \EH\otimes\Odue(t,t) \to 0,$$
while sequence (\ref{rD}) is in fact
$$0 \to \EH\otimes\Odue(t-1,t-1) \to \EH\otimes\Odue(t,t) \to \ED(t) \to 0.$$

Let $E$ be a normalized rank 2 vector bundle on $Q_3$ which is $k$-Buchsbaum, with $k\ge1$. Let $H$ be a general hyperplane section of $Q_3$ defined by the section $x\in H^0(Q_3,\Otre(1))$. 
We denote by $\phi_t:=\phi_{x,t}$ the multiplication map induced by the restriction sequence (\ref{rH}) on the first cohomology groups, i.e.
$$\phi_t \colon H^1(Q_3,E(t-1)) \to H^1(Q_3,E(t))$$
is the multiplication by $x$. 
By the long cohomology sequence we have 
$$\ker(\phi_t) \cong \coker\big(H^0(Q_3,E(t))\to H^0(Q_2,\EH(t))\big)$$
so we have
\begin{equation}\label{dimker}
\dim\ker(\phi_t) \le h^0(Q_2,\EH(t)) \qquad\forall\,t\in\ZZ.
\end{equation}
Taking the composition of $k$ successive such maps we obtain
$$\ker(\phi_t \circ \cdots \circ \phi_{t-k+1}) \subseteq H^1(Q_3,E(t-k));$$
so, if $E$ is $k$-Buchsbaum, we get for each integer $t$ the following equality
\begin{equation}\label{ker}
H^1(Q_3,E(t-k)) = \ker(\phi_t \circ \cdots \circ \phi_{t-k+1})
\end{equation}
since $\phi_t \circ \cdots \circ \phi_{t-k+1} \equiv 0$.

\bigskip

Now we recall some known results on rank 2 vector bundles on a quadric threefold $Q_3$.

\begin{theorem}[Splitting Criterion]\label{hsc}
Let $E$ be a rank 2 vector bundle on a smooth quadric threefold $Q_3$ with first Chern class $c_1$ and first relevant level $\alpha$. If $E$ is arithmetically Cohen-Macaulay, then $E$ splits, unless $0< c_1+2\alpha < 2$.
\end{theorem}
\begin{proof}
See \cite{m1}, Theorem~2.
\end{proof}

\begin{corollary}\label{ACMnonsplit}
The only indecomposable, arithmetically Cohen-Macaulay, normalized, rank 2 vector bundles on a smooth quadric threefold $Q_3$ are stable with $c_1=-1$, $c_2=1$, and $\alpha=1$, i.e.\ they are the spinor bundles.
\end{corollary}
\begin{proof}
By the above splitting criterion, if $E$ is an indecomposable, arithmetically Cohen-Macaulay, normalized, rank 2 vector bundle on $Q_3$, then we must have $c_1 + 2\alpha=1$, so the only possibility is $c_1=-1$ and $\alpha=1$, hence $E$ is stable. Moreover, we have $1-c_2=\chi(E)=0$, that is $c_2=1$. Therefore $E$ is a spinor bundle on $Q_3$ (see \cite{Ot2}).
\end{proof}

\begin{theorem}[Restriction Theorem]\label{rectheo}
Let $E$ be a rank 2 vector bundle on a smooth quadric threefold $Q_3$. If $E$ is stable, then $\EH$, the restriction of $E$ to a general hyperplane section $H$ of $Q_3$, is again stable.
\end{theorem}
\begin{proof}
See \cite{es}, Theorem~1.6.
\end{proof}

\begin{theorem}\label{rectoconic}
If $E$ be a semistable rank 2 vector bundle on a smooth quadric threefold $Q_3$ with $c_1=0$ $($respectively $c_1=-1)$, then $\ED$, the restriction of $E$ to a general conic section $D$ of $Q_3$, is isomorphic to $\OPuno\oplus\OPuno$ $($respectively $\OPuno(-1)\oplus\OPuno(-1))$.
\end{theorem}
\begin{proof}
See \cite{es}, Corollary~1.5.
\end{proof}

\section{Preliminary results}

In the sequel, when we say that a vector bundle is $k$-Buchsbaum we always assume that $k\ge1$, hence excluding arithmetically Cohen-Macaulay bundles.

\subsection{Stable bundles}

\begin{lemma}\label{cohD}
Let $E$ be a semistable, normalized, rank 2 vector bundle on a smooth quadric threefold $Q_3$ and let $D$ be a general conic section of $Q_3$.
Then for every integer $t\ge0$:
$$h^0(D,\ED(t)) = 4t+2c_1+2 \quad\text{and}\quad h^1(D,\ED(t)) = 0.$$
\end{lemma}
\begin{proof}
By Theorem~\ref{rectoconic} $\ED \simeq \OPuno(c_1)\oplus\OPuno(c_1)$, so taking into account that $\OD(t)\simeq\OPuno(2t)$ for all $t\in\ZZ$, we have
$$\ED(t) \simeq \ED\otimes\OD(t) \simeq \OPuno(2t+c_1)\oplus\OPuno(2t+c_1)$$
for all $t\in\ZZ$, then we get
$$h^0(D,\ED(t)) = 2\, h^0(\PP^1,\OPuno(2t+c_1)) = 4t + 2c_1 + 2 \quad\forall\,t\ge 0,$$
and also $h^0(D,\ED(t)) = 0$ for each $t<0$.
The vanishing of the first cohomology follows by Serre duality. 
\end{proof}

\begin{lemma}\label{c2stable}
Let $E$ be a stable, normalized, rank 2 vector bundle on a smooth quadric threefold $Q_3$. Then $c_2>0$.
\end{lemma}
\begin{proof}
By the stability hypothesis it holds $0<a\le\alpha$, where $\alpha=\alpha(E)$ and $a=\alpha(\EH)$, with $H$ a general hyperplane section of $Q_3$. Therefore $h^0(Q_3,E)=0$ and $h^3(Q_3,E)=h^0(Q_3,E(-3-c_1))=0$, hence
$$\chi(E) = h^2(Q_3,E) - h^1(Q_3,E) = h^1(Q_3,E(-3-c_1)) - h^1(Q_3,E) \le 0,$$
since the multiplication map $\phi_t$ is injective for each $t\le a-1$, where $a-1\ge0$. 
We have $\chi(E) = 2-\frac{3}{2}c_2$ if $c_1=0$ and $\chi(E) = 1-c_2$ if $c_1=-1$, so we get $c_2\ge2$ when $c_1=0$, and $c_2\ge1$ when $c_1=-1$.
\end{proof}

\begin{remark}
Observe that the above Lemma is a consequence of Bogomolov's inequality 
(see e.g.\ \cite{HL}, Theorem~7.3.1), but here we have given a proof based on elementary techniques.
\end{remark}

\begin{lemma}\label{h1stable}
Let $E$ be a stable, $k$-Buchsbaum, normalized, rank 2 vector bundle on $Q_3$. Then $h^1(Q_3,E(-1))\ne 0$ if $c_1=0$, and $h^1(Q_3,E) \ne 0$ if $c_1=-1$.
\end{lemma}
\begin{proof}
By Lemma~\ref{c2stable} it holds $c_2>0$, and this implies $\vartheta>0$, where $\vartheta=\frac{3}{2}c_2+\frac{1}{4}$ if $c_1=0$ and $\vartheta=\frac{3}{2}c_2-\frac{1}{2}$ if $c_1=-1$. Moreover $-3/2+\sqrt{\vartheta}>-1$ for each $c_2\ge2$, $c_2$ even, when $c_1=0$, and also $-1+\sqrt{\vartheta}>0$ for each $c_2\ge2$ when $c_1=-1$. So by \cite{BVV}, Theorem~5.4, we get the claim.
\end{proof}

\begin{lemma}\label{van}
Let $E$ be a stable, $k$-Buchsbaum, normalized, rank 2 vector bundle on $Q_3$. Let $H$ be a general hyperplane section of $Q_3$, and let $a$ be the first relevant level of $\EH$. Then $h^1(Q_3,E(t)) = 0$ for each $t \le a-k-1$.
\end{lemma}
\begin{proof}
By the stability hypothesis we have $0<a\le\alpha$, where $\alpha$ is the first relevant level of $E$. Then we have $h^0(Q_2,\EH(t))=0$ for all $t\le a-1$, hence the multiplication map $\phi_t$ is injective for each $t\le a-1$. Therefore, the composition of $k$ successive multiplication maps $\phi_t \circ \cdots \circ \phi_{t-k+1}$ is, at the same time, injective and the zero map for each $t\le a-1$, so we obtain (see (\ref{ker})) $h^1(Q_3,E(t)) = 0$ for all $ t \le a-k-1$.
\end{proof}

\begin{lemma}\label{ak}
Let $E$ be a stable, $k$-Buchsbaum, normalized, rank 2 vector bundle on $Q_3$. Let $H$ be a general hyperplane section of $Q_3$, and let $a$ be the first relevant level of $\EH$. Then
$$a \le \begin{cases} k-1 \quad & \text{if\ \ } c_1=0 \\[1pt] k & \text{if\ \ } c_1=-1 \end{cases}$$
\end{lemma}
\begin{proof}
It is enough to use Lemma~\ref{h1stable} and Lemma~\ref{van}.
\end{proof}

\subsection{Non-stable bundles}

\begin{lemma}\label{non-stable}
Let $E$ be a normalized rank 2 vector bundle on $Q_3$ and $\EH$ its restriction to a general hyperplane section $H$ of $Q_3$. If $E$ is non-split and non-stable, then
\begin{enumerate}
\item $\alpha(\EH)=\alpha(E)=\alpha$,
\item $h^0(Q_3,E(\alpha))=h^0(Q_2,\EH(\alpha))=1$,
\item $h^0(Q_3,E(t))=h^0(Q_3,\Otre(t-\alpha))$ for all $t\le -\alpha-c_1$,
\item $h^0(Q_2,\EH(t))=h^0(Q_2,\Odue(t-\alpha))$ for all $t\le -\alpha-c_1$,
\item for each $t\le -\alpha-c_1$ the following sequence of cohomology groups
$$0 \to H^0(Q_3,E(t-1)) \to H^0(Q_3,E(t)) \to H^0(Q_2,\EH(t)) \to 0$$
is exact,
\item the multiplication map $\phi_t \colon H^1(Q_3,E(t-1)) \to H^1(Q_3,E(t))$ by a general linear form $x\in H^0(Q_3,\Otre(1))$ is injective for each $t\le -\alpha-c_1$.
\end{enumerate}
\end{lemma}
\begin{proof}
By the definition of $\alpha$ and the assumption that $E$ does not split, $E(\alpha)$ has a global section whose zero locus $Z$ has codimension 2 in $Q_3$, hence, by the Serre correspondence, we have for all $t\in\ZZ$ the short exact sequence
$$0 \to \Otre(t-\alpha) \to E(t) \to \shI_Z(t+\alpha+c_1) \to 0$$
which in cohomology gives
$$0 \to H^0(Q_3,\Otre(t-\alpha)) \to H^0(Q_3,E(t)) \to H^0(Q_3,\shI_Z(t+\alpha+c_1)) \to 0$$
and since $t+\alpha+c_1 \le 0$ for every $t\le -\alpha-c_1$ it follows that $H^0(Q_3,E(t)) \cong H^0(Q_3,\Otre(t-\alpha))$ for every $t\le -\alpha-c_1$, in particular we get $h^0(Q_3,E(\alpha))=h^0(Q_3,\Otre)=1$ (these are (1), (2), and (3)).
Now we consider the short exact sequence obtained by tensoring the previous one with $\OH$
$$0 \to \Odue(t-\alpha) \to \EH(t) \to \shI_{Z\cap H}(t+\alpha+c_1) \to 0$$
from which we get
$$0 \to H^0(Q_2,\Odue(t-\alpha)) \to H^0(Q_2,\EH(t)) \to H^0(Q_2,\shI_{Z\cap H}(t+\alpha+c_1)) \to 0$$
so we have $H^0(Q_2,\EH(t)) \cong H^0(Q_2,\Odue(t-\alpha))$ for every $t\le -\alpha-c_1$ (this is (4)), in particular $h^0(Q_2,\EH(\alpha))=h^0(Q_2,\OH)=1$. 
Moreover from the exact sequence
$$0 \to \Otre(-1) \to \Otre \to \Odue \to 0$$
we obtain in cohomology for all $t\in\ZZ$
$$0 \to H^0(Q_3,\Otre(t-1)) \to H^0(Q_3,\Otre(t)) \to H^0(Q_2,\Odue(t)) \to 0,$$
hence we have proved (5). Moreover, (6) follows from (5) through the long cohomology sequence.
\end{proof}

\begin{lemma}\label{van-non-stable}
Let $E$ be a non-stable, $k$-Buchsbaum, normalized, rank 2 vector bundle on $Q_3$ with first relevant level $\alpha$. Then $h^1(Q_3,E(t))=0$ for all $t \le -\alpha-c_1-k$.
\end{lemma}
\begin{proof}
By Lemma~\ref{non-stable}(6) the multiplication map $\phi_t$ is injective for each $t\le -\alpha-c_1$. Therefore, the composition $\phi_t \circ \cdots \circ \phi_{t-k+1}$ of $k$ successive multiplication maps is injective and also the zero map for each $t\le -\alpha-c_1$, so we get the claim.
\end{proof}

\begin{proposition}
There is no normalized rank 2 vector bundle $E$ on $Q_3$ which is properly $k$-Buchsbaum  with $\alpha\le 1-k$.
\end{proposition}
\begin{proof}
Assume $E$ properly $k$-Buchsbaum with $\alpha\le 1-k$ and $c_1\in\{0,-1\}$. Since $k\ge1$ we get $\alpha\le0$, so $E$ is non-stable. 
By Lemma~\ref{van-non-stable} we have $h^1(Q_3,E(t))=0$ for all $t\le-\alpha-c_1-k$.
On the other hand, the hypothesis $\alpha\le 1-k$ implies $-\frac{3+c_1}{2} < -\alpha-c_1-k \le -\alpha-c_1-1$, so by \cite{BVV}, Theorem~4.3, we must have $h^1(Q_3,E(-\alpha-c_1-k))\ne0$, which is in contradiction with the above vanishing.
\end{proof}

\begin{corollary}\label{alpha-ns}
Let $E$ be a non-stable, normalized, rank 2 vector bundle on $Q_3$. If $E$ is properly $k$-Buchsbaum, then $2-k \le \alpha \le 0$.
\end{corollary}

\begin{corollary}\label{nsaB}
Every properly arithmetically Buchsbaum rank 2 vector bundle on $Q_3$ is stable.
\end{corollary}

\begin{proposition}\label{EDns}
Let $E$ be a non-stable, normalized, rank 2 vector bundle on a smooth quadric threefold $Q_3$ with first relevant level $\alpha$. Then $\ED$, the restriction of $E$ to a general conic section $D$ of $Q_3$, is isomorphic to $\OD(-\alpha)\oplus\OD(\alpha+c_1)$.
\end{proposition}
\begin{proof}
Let $Z$ be the zero locus of  a non-zero global section of $E(\alpha)$, 
which is a subscheme of $Q_3$ of pure dimension 1, then we have the exact sequence
$$0 \to \Otre \to E(\alpha) \to \shI_Z(2\alpha+c_1) \to 0.$$
If we tensor the above sequence by $\OD(-\alpha)$, where $D$ is a general conic section of  $Q_3$, hence not meeting $Z$, we get
$$0 \to \OD(-\alpha) \to \ED \to \OD(\alpha+c_1) \to 0.$$
By hypothesis $2\alpha+c_1 \le 0$, therefore we have
\begin{align*}
\Ext^1(\OD(\alpha+c_1), \OD(-\alpha)) & \cong \Ext^1(\OD,\OD(-2\alpha-c_1)) \\
& \cong H^1(D,\OD(-2\alpha-c_1)) \\ 
& \cong H^0(D,\OD(2\alpha+c_1-1)) = 0
\end{align*}
so the above sequence splits, that is $\ED \simeq \OD(-\alpha)\oplus\OD(\alpha+c_1)$.
\end{proof}

\begin{lemma}\label{cohDns}
Let $E$ be a non-stable, normalized, rank 2 vector bundle on a smooth quadric threefold $Q_3$ with first relevant level $\alpha$ and let $D$ be a general conic section of $Q_3$.
Then for every integer $t\ge0$:
$$h^0(D,\ED(t)) = \begin{cases} 0 & \text{for\ \ } t \le \alpha-1 \\
2t-2\alpha+1 & \text{for\ \ } \alpha\le t\le -\alpha-c_1-1 \\
4t+2c_1+2 & \text{for\ \ } t\ge-\alpha-c_1 \end{cases}$$
\end{lemma}
\begin{proof}
By Proposition~\ref{EDns} we have $\ED \simeq \OD(-\alpha)\oplus\OD(\alpha+c_1)$, so, taking into account that $\OD(t)\simeq\OPuno(2t)$ for all $t\in\ZZ$, with a simple computation we get the claim.
\end{proof}

\section{Arithmetically Buchsbaum rank 2 bundles}

In the present section we give the proof of Theorem~\ref{main}, but dividing it into two statements: Theorem~\ref{aBonQ3} and Theorem~\ref{aBonQn}.
\\
First, we analyze what happens on a smooth quadric threefold $Q_3\subset\Pfour$, starting with
a technical lemma which we need in the following.

\begin{lemma}\label{coho}
Let $E$ be an arithmetically Buchsbaum rank 2 vector bundle on $Q_3$. If $h^2(Q_3,E(t-1))=0$, then $H^1(Q_3,E(t)) \cong H^1(Q_2,E_H(t))$ through the restriction map and moreover the multiplication map 
$$H^1(Q_2,\EH(t))\to H^1(Q_2,\EH(t+1))$$
is the zero map.
\end{lemma}
\begin{proof}
Consider the following commutative diagram
$$\begin{matrix}
H^1(Q_3,E(t)) & \xrightarrow{\cong} & H^1(Q_2,\EH(t)) \\
\noalign{\vspace{4pt}}
{}\quad\downarrow & & {}\quad\downarrow  \\
\noalign{\vspace{4pt}}
H^1(Q_3,E(t+1)) & \to & H^1(Q_2,\EH(t+1)) 
\end{matrix}$$
where the horizontal maps are the restriction maps obtained from restriction sequence (\ref{rH}), 
while the vertical maps are the multiplication maps by a general linear form not defining the hyperplane section $H\cong Q_2$, so by the hypotheses it follows that the left vertical map is the zero map.
\end{proof}

Now we are able to give the full classification of arithmetically Buchsbaum rank 2 vector bundles on a quadric threefold $Q_3$.

\begin{theorem}\label{aBonQ3}
Let $E$ be an arithmetically Buchsbaum, normalized, rank 2 vector bundle on $Q_3$. 
Then $E$ is one of the following:
\begin{enumerate}
\item $E$ is a split bundle;
\item $E$ is stable with $c_1=-1$, $c_2=1$, i.e.\ $E$ is a spinor bundle;
\item $E$ is stable with $c_1=-1$, $c_2=2$, i.e.\ $E$ is associated to two skew lines or to a double line;
\item $E$ is stable with $c_1=-1$, $c_2=3$, and $H^0(Q_3,E(1))=0$, i.e.\ $E$ is associated to a smooth elliptic curve of degree $7$ in $Q_3\subset\Pfour$.
\end{enumerate}
Therefore $E$ is either arithmetically Cohen-Macaulay, cases (1) and (2), or properly arithmetically Buchsbaum, cases (3) and (4), with only one non-zero first cohomology group.
\end{theorem}
\begin{proof}
Let $E$ be an arithmetically Buchsbaum, normalized, rank 2 vector bundle on $Q_3$ and let $H$ and $D$ be general hyperplane and conic sections of $Q_3$. Since the rank 2 arithmetically Cohen-Macaulay 
vector bundles on $Q_3$ are either split or spinor bundles (see Theorem~\ref{acm}), we can assume that $E$ is properly arithmetically Buchsbaum. 
We set $\alpha=\alpha(E)$ and $a=\alpha(\EH)$.
By Corollary~\ref{nsaB} $E$ must be stable, and moreover, by Lemma~\ref{ak} and Theorem~\ref{rectheo} there is no stable properly arithmetically Buchsbaum bundle with $c_1=0$.
Therefore we may assume, always by Lemma~\ref{ak}, that $E$ is stable with $c_1=-1$ and $a=1$. 
We have $h^1(Q_3,E(t))=0$ for all $t\le -1$, by Lemma~\ref{van}, and $h^2(Q_3,E(t))=0$ for all $t\ge -1$, by Serre duality. 
Moreover, by Lemma~\ref{h1stable} $h^1(Q_3,E)\ne0$. Thanks to the vanishing of $2$-cohomology for $t\ge-1$, the multiplication map $H^1(Q_2,\EH(t))\to H^1(Q_2,\EH(t+1))$, by Lemma~\ref{coho}, is the zero map for every $t\ge0$. From the restriction sequence (\ref{rD}) we get in cohomology the exact sequence 
$$0 \to H^1(Q_2,\EH(t)) \to H^1(D,\ED(t))=0$$
for all $t\ge1$ (taking into account Lemma~\ref{cohD}), so we obtain $h^1(Q_2,\EH(t))=0$ for all $t\ge1$. Now from restriction sequence (\ref{rH}) we get in cohomology the exact sequence
$$0 \to H^1(Q_3,E(t)) \to H^1(Q_2,\EH(t))=0$$
for every $t\ge1$, hence $h^1(Q_3,E(t))=0$ for all $t\ge 1$.
Therefore $E$ must have $h^1(Q_3,E(t))=0$ for all $t\ne0$. 
\\
Since $\alpha\ge a=1$, we have either $\alpha=1$ or $\alpha>1$; in this last event we have  $h^1(Q_3,E(\alpha-2))\ne0$, because of \cite{BVV}, Theorem 5.2, and therefore we must have $\alpha=2$.
So we have to analyze two possibilities: $\alpha=1$ and $\alpha=2$. In both cases it holds  $6-2c_2=\chi(E(1))=h^0(Q_3,E(1))\ge0$, that is $c_2\le3$, therefore, being $E$ properly arithmetically Buchsbaum, we obtain $c_2=2$ if and only if $\alpha=1$, and $c_2=3$ if and only if $\alpha=2$.
\\
If $c_2=2$ and $\alpha=1$, then we get $h^1(Q_3,E)=1$ and $h^0(Q_3,E(1))=2$. 
By the Serre correspondence we have an exact sequence on $Q_3$ like the following
$$0\to \OO \to E(1) \to \shI_Z(1) \to 0$$
where $Z$ is a non-empty locally complete intersection curve of degree $2$ with  $\omega_Z\simeq\OO_Z(-2)$, arithmetic genus $-1$, and $h^1(\Pfour,\shI_{Z,\Pfour}(t))=h^1(Q_3,\shI_Z(t))=0$ for all $t\ne0$ and $h^1(\Pfour,\shI_{Z,\Pfour})=h^1(Q_3,\shI_Z)=1$, so $E$ is a vector bundle associated to two skew lines or a double line.
\\
If $c_2=3$ and $\alpha =2$, then we get $h^1(Q_3,E)=2$ and $h^0(Q_3,E(1))=0$.
By the Serre correspondence $E(2)$ fits into an extension of the following type
$$0\to \OO \to E(2) \to \shI_C(3) \to 0$$
where $C$ is a non-empty locally complete intersection curve of degree $7$ and arithmetic genus $1$, with $\omega_C\simeq\OO_C$ and $h^1(\Pfour,\shI_{C,\Pfour}(t))=h^1(Q_3,\shI_C(t))=0$ for all $t\ne1$ and $h^1(\Pfour,\shI_{C,\Pfour}(1))=h^1(Q_3,\shI_C(1))=2$, moreover $E(2)$ is generated by global sections by Castelnuovo-Mumford criterion (see Theorem~\ref{CM}), so $\shI_C(3)$ is globally generated too, hence the zero locus of a general section of $E(2)$ is smooth. Therefore $E$ is a vector bundle associated to a smooth elliptic curve of degree $7$ in $Q_3\subset\Pfour$.
\end{proof}

Now we can state and prove the classification of arithmetically Buchsbaum rank 2 vector bundles on a quadric hypersurface $Q_n\subset\PP^{n+1}$, with $n\ge4$.

\begin{theorem}\label{aBonQn}
The only indecomposable, arithmetically Buchsbaum, normalized, rank 2 vector bundle $F$ on $Q_n$, $n\ge4$, are the following:
\begin{enumerate}
\item for $n=5$, $F$ is a Cayley bundle, i.e.\ $F$ is a bundle with $c_1=-1$, $c_2=2$;
\item for $n=4$, $F$ is a spinor bundle 
or it has $c_1=-1$, $c_2=(1,1)$, i.e.\ $F$ is the restriction of a Cayley bundle to $Q_4$.
\end{enumerate}
Moreover, for $n\ge6$ no such bundle exists.
\end{theorem}
\begin{proof}
Let $F$ be a rank $2$ vector bundle on $Q_n$ as in the statement. Take a general $3$-dimensional linear section $Q_3$ of $Q_n$ and set $E:=F\vert_{Q_3}$. By the hypothesis $E$ is an indecomposable, arithmetically Buchsbaum, normalized, rank $2$ vector bundle on the quadric threefold $Q_3$, so the only possibilities for $E$ are listed in Theorem~\ref{aBonQ3}. We analyze each case separately. \\
(a) If $E$ is a spinor bundle on $Q_3$, then $E$ extends to a spinor bundle on $Q_4$, i.e.\ a stable bundle with $c_1=-1$ and $c_2=(1,0)$ or $c_2=(0,1)$, but does not extend to any further bundle on $Q_n$ for $n\ge5$ (see \cite{Ot2}, Theorem~2.1).
\\
(b) If $E$ is stable with $c_1=-1$, $c_2=2$, then $E$ extends to a vector bundle on $Q_4$ with Chern classes $c_1=-1$, $c_2=(1,1)$, and even to one on $Q_5$ with $c_1=-1$, $c_2=2$ (a Cayley bundle), but to no further bundle on any $Q_n$, $n\ge6$ (see Definition~\ref{Cb1} and Theorem~\ref{Cb2} and \ref{Cb3}).
\\
(c) If $E$ is stable with $c_1=-1$, $c_2=3$, and $a=2$, then $E$ does not extend to any arithmetically Buchsbaum bundle $F$ on $Q_4$. In fact, assume there exists an extension $F$ of the bundle $E$ on $Q_4$, i.e.\ $E=F\vert_{Q_3}$. By the proof of Theorem~\ref{aBonQ3} we know that $h^0(Q_3,E(2))\ne0$, while $h^1(Q_3,E(t))=0$ for all $t\ne0$ and $h^1(Q_3,E)=2$.
By the assumption, the two vector bundle $F$ and $E$ fit into a restriction sequence like 
$$0 \to F(-1) \to F \to E \to 0,$$
so we get in cohomology the exact sequence
$$0 \to H^1(Q_4,F(t-1)) \to H^1(Q_4,F(t)) \to 0$$
for all $t\le-1$ and $t\ge 1$, since $h^0(Q_4,F(t))=0$ for all $t\le 1$. It follows that $h^1(Q_4,F(t))=h^3(Q_4,F(t))=0$ for all $t\in\ZZ$. Similarly, we get $h^2(Q_4,F(t))=0$ for all $t\le-3$ and $t\ge0$, and also $h^2(Q_4,F(-1))=h^2(Q_4,F(-2))=2$. Therefore we have $h^1(Q_4,F(1))=h^2(Q_4,F)=h^3(Q_4,F(-1))=0$ and $h^4(Q_4,F(-2))=h^0(Q_4,F(-1))=0$, so by the  Castelnuovo-Mumford criterion (see Theorem~\ref{CM}) $F$ is $2$-regular, hence $F(2)$ is generated by global sections and therefore the zero locus of a general global section of $F(2)$ is a smooth surface $S$ of degree $7$ and we have the exact sequence
$$0 \to \OO_{Q_4} \to F(2) \to \shI_S(3) \to 0.$$
Since $\det(F(2))\simeq\OO_{Q_4}(3)$ and $\omega_{Q_4}\simeq\OO_{Q_4}(-4)$, the adjuction formula gives $\omega_S\simeq\OO_S(-1)$. Thus $S$ is an anticanonically embedded del Pezzo surface. By the above exact sequence we see that $h^1(\PP^4,\shI_S(t))=h^1(Q_4,F(1))=0$ for all $t\in\ZZ$. Using Riemann-Roch on the surface $S$ we obtain $h^0(S,\OO_S(1))=\chi(\OO_S(1)) = (\OO_S(1)\cdot\OO_S(2))/2 + \chi(\OO_S) = 7 + 1 = 8$. On the other hand, using the structure sequence $0 \to \shI_S \to \OO_{Q_4} \to \OO_S \to 0$, we have $h^0(S,\OO_S(1)) \le h^0(\PP^4,\OO_{Q_4}(1)) = 6$, which is absurd. Therefore, there exists no arithmetically  Buchsbaum bundle $F$ on $Q_n$, $n\ge4$, such that its restriction to a general 3-dimensional linear section $Q_3$ is a stable bundle with $c_1=-1$, $c_2=3$, and $a=2$.
\end{proof}

\begin{remark}
Notice that in Theorem~\ref{aBonQn} the hypothesis that $F$ is an arithmetically Buchsbaum bundle can be weakened to the following one: we can ask that $F$ is a bundle with \emph{1-Buchsbaum first cohomology}, meaning that for every integer $q$ such that $3\le q\le n$ it holds
$$\mm \cdot H^1_*(Q',F\vert_{Q'}) = 0,$$
where $Q'$ is a general $q$-dimensional linear section of $Q_n$. \\
Obviously, by Serre duality, for a rank $2$ vector bundle $E$ on a quadric threefold $Q_3$ the two conditions:\\
${}$\quad a) \lq\lq$E$ is arithmetically Buchsbaum\rq\rq\\
${}$\quad b) \lq\lq$E$ has 1-Buchsbaum first cohomology \rq\rq\\
are equivalent.
Instead, on an $n$-dimensional quadric $Q_n$, $n\ge4$, condition a) implies condition b), but a priori the converse is not true.
\end{remark}

\section{Boundedness for $c_2$ of $k$-Buchsbaum bundles}

In this section we investigate the 0-cohomology of the restriction to a general hyperplane section $\EH$ of a rank 2 vector bundle $E$ on a quadric threefold $Q_3\subset\Pfour$, in order to establish some bounds on the second Chern class $c_2$ of a $k$-Buchsbaum bundle on $Q_3$, both in the stable and in the non-stable case.

We start with the stable case.

\begin{lemma}\label{lemma_h0}
Let $E$ be a stable, normalized, rank 2 vector bundle on $Q_3$. Let $H$ be a general hyperplane section of $Q_3$, and let $a$ be the first relevant level of $\EH$. Then for each $t \ge 0$
$$h^0(Q_2,\EH(t)) \le 2t(t+2+c_1),$$
Moreover, for $t\ge a$ we have the better bound
$$h^0(Q_2,\EH(t)) \le 2t(t+2+c_1) - 2(a-1)(a+1+c_1).$$
\end{lemma}
\begin{proof}
Let $D$ be a general conic section of $Q_3$. Then we have the exact sequence
$$0 \to \EH(i-1) \to \EH(i) \to \ED(i) \to 0$$
for each $i\in\ZZ$ (where $\EH(i)$ means $\EH\otimes\Odue(i,i)$ as above), so in cohomology we get the exact sequence
$$0 \to H^0(Q_2,\EH(i-1)) \to H^0(Q_2,\EH(i)) \to H^0(D,\ED(i))$$
and therefore
$$h^0(Q_2,\EH(i)) - h^0(Q_2,\EH(i-1)) \le h^0(D,\ED(i))$$
for every $i\in\ZZ$.
Fix an integer $t\ge a$. Then we obtain, by Lemma~\ref{cohD},
\begin{align*}
h^0(Q_2,\EH(t)) & = \sum_{i=a}^t \Big( h^0(Q_2,\EH(i)) - h^0(Q_2,\EH(i-1)) \Big) \le \\
& \le \sum_{i=a}^t h^0(D,\ED(i)) = \sum_{i=a}^t (4i + 2c_1 + 2) = \\[4pt]
& = 2t(t+2+c_1) - 2(a-1)(a+1+c_1).
\end{align*}
Taking into account that $a\ge1$ and $h^0(Q_2,\EH(t)) = 0$ for $t<a$, we get the claim.
\end{proof}

\begin{remark}
Lemma~\ref{lemma_h0} holds more generally for any $\Odue(1,1)$-stable, normalized, rank 2 vector bundle $F$ on $Q_2$ with first relevant level, with respect to $\Odue(1,1)$, $a=\alpha(F)$.
\end{remark}

\begin{proposition}\label{stable-case}
Fix integers $k\ge 1$ and $c_1\in\{0,-1\}$. Let $\Ss(k,c_1)$ be the family of all stable, properly $k$-Buchsbaum, rank 2 vector bundles $E$ on $Q_3$ with first Chern class $c_1$. Set $C_s(k,c_1):=\{c_2(E)\in\ZZ \mid E\in\Ss(k,c_1)\}$. Then $C_s(k,c_1)$ is finite and $\Ss(k,c_1)$ is bounded. 
In particular
\begin{align*}
& C_s(k,0) = \{ c_2 \in 2\ZZ \mid 2 \le c_2 \le 2k(k-1)(2k+5)/3 \}, \\[4pt]
& C_s(k,-1) = \{ c_2\in\ZZ \mid 1 \le c_2 \le k(k+1)(2k+4)/3+1 \}.
\end{align*}
\end{proposition}
\begin{proof}
Take $E\in\Ss(k,c_1)$, then $E$ is a stable, properly $k$-Buchsbaum, rank 2 vector bundle on $Q_3$ with first Chern class $c_1$. The stability of $E$ implies that $c_2=c_2(E)>0$ (see Lemma~\ref{c2stable}). \\
Let $H\cong Q_2$ be a general hyperplane section of $Q_3$, then $\EH$ is stable (with respect to the line bundle $\Odue(1,1)$) by Theorem~\ref{rectheo}.
\\
First assume $c_1=0$ and set 
$$\eta=\eta(E):= \sum_{i=0}^{k-1} h^0(Q_2,\EH(i)).$$
For each $i\in\ZZ$ we have by (\ref{dimker})
$$\dim\ker(\phi_i) \le h^0(Q_2,\EH(i)),$$
where $\phi_i\colon H^1(Q_3,E(i-1)) \to H^1(Q_3,E(i))$ is the multiplication map by an element $x\in H^0(Q_3,\Otre(1))$ defining the hyperplane section $H$.
By hypothesis $E$ is a $k$-Buchsbaum bundle, so
$$h^1(Q_3,E(-1)) = \dim\ker(\phi_{k-1}\circ\cdots\circ\phi_0) \le \eta.$$
Now 
$$-\frac{c_2}{2} = \chi(E(-1)) = - h^1(Q_3,E(-1)) + h^2(Q_3,E(-1)) \ge - h^1(Q_3,E(-1))$$
since $h^0(Q_3,E(-1))=0$ and $h^3(Q_3,E(-1))=h^0(Q_3,E(-2))=0$, being $\alpha>0$, so we obtain
$$c_2 \le 2\, h^1(Q_3,E(-1)) \le 2 \eta.$$
Moreover, by Lemma~\ref{lemma_h0}, we have
$$\eta = \sum_{i=0}^{k-1} h^0(Q_2,\EH(i)) \le \sum_{i=0}^{k-1} 2i(i+2)  = \frac{1}{3}k(k-1)(2k+5).$$
Therefore we obtain
$$2\le c_2 \le 2k(k-1)(2k+5)/3.$$
Now assume $c_1=-1$ and set
$$\eta=\eta(E):= \sum_{i=1}^{k} h^0(Q_2,\EH(i)).$$
In this case we have, under the hypothesis that $E$ is $k$-Buchsbaum, 
$$h^1(Q_3,E) = \dim\ker(\phi_{k}\circ\cdots\circ\phi_1) \le \eta.$$
It holds
$$1-c_2 = \chi(E) = - h^1(Q_3,E) + h^2(Q_3,E)) \ge - h^1(Q_3,E)$$
since $h^0(Q_3,E)=0$ and $h^3(Q_3,E)=h^0(Q_3,E(-2))=0$, being $\alpha>0$, so we obtain
$$c_2 \le h^1(Q_3,E) + 1 \le \eta + 1.$$
Moreover, by Lemma~\ref{lemma_h0}, we have
$$\eta = \sum_{i=1}^{k} h^0(Q_2,\EH(i)) \le \sum_{i=1}^{k} 2i(i+1)  = \frac{1}{3}k(k+1)(2k+4).$$
Therefore we obtain
$$1\le c_2 \le k(k+1)(2k+4)/3+1.$$
Finally, the boundedness of $\Ss(k,c_1)$ follows from the finiteness of $C_s(k,c_1)$ and the boundedness of the family of all stable vector bundles with fixed rank and Chern classes (see \cite{HL}, Theorem~3.3.7).
\end{proof}

\begin{remark}
Notice that the above result can be improved if we take into account the following facts:
the multiplication maps $\phi_i$ are injective for all $i\le a-1$, and also, by Lemma~\ref{ak}, it holds $a\le k-1$ if $c_1=0$ and $a\le k$ if $c_1=-1$, so we can set
$$\eta=\sum_{i=a}^{k-1} h^0(Q_2,\EH(i)) \;\;\text{if } c_1=0, \quad\text{and}\quad \eta=\sum_{i=a}^{k} h^0(Q_2,\EH(i)) \;\;\text{if } c_1=-1,$$
hence using the bound for $h^0(Q_2,\EH(i))$ depending on $a$ of Lemma~\ref{lemma_h0} we obtain the following better upper bounds:
$$c_2 \le \begin{cases} 2(2k+4a+1)(k-a)(k-a+1)/3 &  \text{if } c_1=0, \\[5pt]
2(k+2a)(k-a+1)(k-a+2)/3 + 1 & \text{if } c_1=-1, \end{cases}$$
that are dependent on $k$ and on $a=\alpha(\EH)$ (which is not so easy to  compute).
\end{remark}

\begin{remark}
Applying Proposition~\ref{stable-case} in the case of $1$-Buchsbaum rank 2 vector bundles on $Q_3$ we obtain:
$$C_s(1,0) = \emptyset \quad\text{and}\quad C_s(1,-1) = \{c_2\in\ZZ \mid 1\le c_2 \le 5\},$$
so we find again the fact that there exists no stable properly 1-Buchsbaum bundle with first Chern class $c_1=0$. \\
Moreover, in the case of $2$-Buchsbaum bundles, we obtain:
$$C_s(2,0) = \{c_2\in2\ZZ \mid 2\le c_2 \le 12\} \quad\text{and}\quad C_s(2,-1) = \{c_2\in\ZZ \mid 1\le c_2 \le 17\}.$$
Observe that, thanks to Theorem~\ref{aBonQ3}, we know that it holds $C_s(1,-1)=\{1,2,3\}$, therefore the above upper bound on $c_2$ is not sharp.
\end{remark}

Now we consider the non-stable case.

\begin{remark}
Observe, that, by Corollary~\ref{nsaB}, there exist no rank 2 vector bundle on $Q_3$ which is non-stable and properly 1-Buchsbaum.
\end{remark}

\begin{lemma}\label{lemma_h0_ns}
Let $E$ be a non-stable, normalized, rank 2 vector bundle on $Q_3$ with first relevant level $\alpha$ and let $H$ be a general hyperplane section of $Q_3$. Then for each $t \ge -\alpha - c_1 +1$
$$h^0(Q_2,\EH(t)) \le 2t(t+2+c_1)+2\alpha^2 + 2c_1\alpha +1+c_1.$$
\end{lemma}
\begin{proof}
Let $D$ be a general conic section of $Q_3$. Then we have, see the proof of Lemma~\ref{lemma_h0}, 
$$h^0(Q_2,\EH(i)) - h^0(Q_2,\EH(i-1)) \le h^0(D,\ED(i))$$
for every $i\in\ZZ$. Fix an integer $t\ge -\alpha - c_1 +1$. Then, using Lemma~\ref{cohDns}, we obtain
\begin{align*}
h^0(Q_2,\EH(t)) & \le h^0(Q_2,\EH(-\alpha-c_1)) + \sum_{i=-\alpha-c_1+1}^t h^0(D,\ED(i)) = \\[4pt]
& = (-2\alpha-c_1+1)^2 + \sum_{i=-\alpha-c_1+1}^t (4i+2c_1+2) = \\[4pt]
& = 2t(t+2+c_1)+2\alpha^2 + 2c_1\alpha +1+c_1. \qedhere
\end{align*}
\end{proof}

\begin{proposition}\label{non-stable-case}
Fix integers $k\ge 2$ and $c_1\in\{0,-1\}$. Let $\N(k,c_1)$ be the family of all non-stable, properly $k$-Buchsbaum, rank 2 vector bundles $E$ on $Q_3$ with first Chern class $c_1$. Set $C_{ns}(k,c_1):=\{c_2(E)\in\ZZ \mid E\in\N(k,c_1)\}$. Then $C_{ns}(k,c_1)$ is finite and $\N(k,c_1)$ is bounded. 
In particular
\begin{align*}
& C_{ns}(k,0) = \{ c_2 \in 2\ZZ \mid -2(k-2)^2  < c_2 \le 2(k-1)(k+1)(2k+3)/3 \}, \\[4pt]
& C_{ns}(k,-1) = \{ c_2\in\ZZ \mid -2(k-1)(k-2) < c_2 \le 2(k-1)(k^2+4k+6)/3+1 \}.
\end{align*}
\end{proposition}
\begin{proof}
Let $\N_\alpha(k,c_1)$ be the subset of $\N(k,c_1)$ containing all non-stable, properly $k$-Buchsbaum, rank 2 vector bundles $E$ with $c_1(E)=c_1$ and $\alpha(E)=\alpha$.
The mininal curve corresponding to any non-zero global section of $E(\alpha)$ has degree $\delta=c_2 + 2 c_1 \alpha + 2 \alpha^2 >0$, so we have $c_2 > -2 c_1 \alpha - 2 \alpha^2$, which is a lower bound for $c_2$ of any bundle $E\in\N_\alpha(k,c_1)$.
\\
Let $H$ be a general hyperplane section of $Q_3$, then $\alpha(\EH)=\alpha$ by Lemma~\ref{non-stable}(1). 
\\
First assume $c_1=0$. With the same reasoning as in the proof of Proposition~\ref{stable-case} we have
$$h^1(Q_3,E(-1)) = \dim\ker(\phi_{k-1}\circ\cdots\circ\phi_0) \le \eta,$$
where
$$\eta := \sum_{i=-\alpha+1}^{k-1} h^0(Q_2,\EH(i)),$$
since the multiplication maps $\phi_t$ are injective for every $t\le-\alpha$ (see Lemma~\ref{non-stable}(6)), and also $-\alpha+1\le k-1$ by Corollary~\ref{alpha-ns}.
Moreover
\begin{align*}
-\frac{c_2}{2} = \chi(E(-1)) ={} & h^0(Q_3,E(-1)) - h^1(Q_3,E(-1)) +{} \\
& {} + h^2(Q_3,E(-1)) - h^3(Q_3,E(-1)) \ge - h^1(Q_3,E(-1))
\end{align*}
since $h^0(Q_3,E(-1))-h^3(Q_3,E(-1))=h^0(Q_3,E(-1))-h^0(Q_3,E(-2))\ge0$, so we obtain
$$c_2 \le 2\, h^1(Q_3,E(-1)) \le 2 \eta.$$
By Lemma~\ref{lemma_h0_ns} we have
\begin{align*}
\eta & = \sum_{i=-\alpha+1}^{k-1} h^0(Q_2,\EH(i)) \le \sum_{i=-\alpha+1}^{k-1} \Big(2i(i+2)+2\alpha^2+1\Big)  = \\[5pt]
& = \frac{1}{3}k(k-1)(2k+5) + \frac{1}{3}\alpha(\alpha-1)(2\alpha-7) + (k-1+\alpha)(2\alpha^2+1) = \\[5pt]
& = \frac{1}{3}(k-1+\alpha)\Big[(k+1)(2k+3) + \alpha(8\alpha-7-2k)\Big].
\end{align*}
Therefore we obtain
$$-2\alpha^2  < c_2(E) \le 2(k-1+\alpha)\big[(k+1)(2k+3) + \alpha(8\alpha-7-2k)\big]/3$$
for each bundle $E\in\N_\alpha(k,0)$.
Now, by Corollary~\ref{alpha-ns}, it holds
$$\N(k,0) = \bigsqcup_{2-k\le\alpha\le0} \N_\alpha(k,0),$$
and also 
\begin{align*}
\frac{1}{3}k(k-1)(2k+5) & + \frac{1}{3}\alpha(\alpha-1)(2\alpha-7) + (k-1+\alpha)(2\alpha^2+1) \le \\[5pt]
& \le \frac{1}{3}k(k-1)(2k+5) + (k-1)(2(k-2)^2 + 1) = \\[5pt]
& = \frac{1}{3}(k-1)(8k^2-19k+27)
\end{align*}
since
$$\frac{1}{3}\alpha(\alpha-1)(2\alpha-7) + \alpha(2\alpha^2+1) \le 0\quad\text{and}\quad \alpha^2 \le (k-2)^2$$
for $2-k\le\alpha\le0$;
so we can say that every vector bundle in $\N(k,0)$ has second Chern class $c_2$ such that:
$$c_2\in2\ZZ \quad\text{and}\quad -2(k-2)^2  < c_2 \le 2(k-1)(8k^2-19k+27)/3.$$
Now assume $c_1=-1$. In this case we have
$$h^1(Q_3,E) = \dim\ker(\phi_{k}\circ\cdots\circ\phi_1) \le \eta,$$where
$$\eta := \sum_{i=-\alpha+2}^{k} h^0(Q_2,\EH(i)),$$
since the multiplication maps $\phi_t$ are injective for every $t\le-\alpha+1$,
and also $-\alpha+2\le k$ by Corollary~\ref{alpha-ns}.
Moreover
$$1-c_2 = \chi(E) = h^0(Q_3,E) - h^1(Q_3,E) + h^2(Q_3,E) - h^3(Q_3,E) \ge - h^1(Q_3,E)$$
since $h^0(Q_3,E)-h^3(Q_3,E)=h^0(Q_3,E)-h^0(Q_3,E(-2))\ge0$, so we obtain
$$c_2 \le h^1(Q_3,E) + 1 \le \eta + 1.$$
By Lemma~\ref{lemma_h0_ns} we have
\begin{align*}
\eta & = \sum_{i=-\alpha+2}^{k} h^0(Q_2,\EH(i)) \le \sum_{i=-\alpha+2}^{k} \Big(2i(i+1)+2\alpha^2-2\alpha\Big)  = \\[5pt]
& {}= \frac{2}{3}k(k+1)(k+2) + \frac{2}{3}(\alpha-1)(\alpha-2)(\alpha-3) + (k-1+\alpha)(2\alpha^2-2\alpha) \\[5pt]
& = \frac{2}{3}(k-1+\alpha)\Big[k^2+4k+6 + \alpha(4\alpha-8-k)\Big].
\end{align*}
Therefore we obtain
$$2\alpha-2\alpha^2  < c_2(E) \le 2(k-1+\alpha)\big[k^2+4k+6 + \alpha(4\alpha-8-k)\big]/3+1$$
for each bundle $E\in\N_\alpha(k,-1)$. Moreover, we have 
\begin{align*}
\frac{2}{3}k(k+1)(k+2) & + \frac{2}{3}(\alpha-1)(\alpha-2)(\alpha-3) + (k-1+\alpha)(2\alpha^2-2\alpha) < \\[5pt]
& < \frac{2}{3}k(k+1)(k+2) + 2(k-1)(k-2) = \\[5pt]
& = \frac{2}{3}k(4k^2-9k+17)-4
\end{align*}
since
$$\frac{2}{3}(\alpha-1)(\alpha-2)(\alpha-3) + \alpha(2\alpha^2-2\alpha) < 0\quad\text{and}\quad \alpha^2-2\alpha \le 2(k-1)(k-2)$$
for $2-k\le\alpha\le0$; so we can say that every vector bundle in $\N(k,-1)$ has second Chern class $c_2$ such that:
$$-2(k-1)(k-2)  < c_2 \le \frac{2}{3}k(4k^2-9k+17)-4.$$
Finally, let $\F(c_1,c_2,\alpha)$ be the family of all non-stable, rank 2 vector bundles $E$ on $Q_3$ with $c_1(E)=c_1$, $c_2(E)=c_2$, and $\alpha(E)=\alpha$. By Corollary~\ref{alpha-ns} it holds
$$\N(k,c_1) \subseteq \bigcup_{\scriptstyle 2-k\le\alpha\le0 \atop \scriptstyle c_2\in C_{ns}(k,c_1)} \F(c_1,c_2,\alpha).$$
The boundedness of each family $\F(c_1,c_2,\alpha)$ follows from \cite{HL}, Theorem~1.7.8 (in the Kleiman criterion choose all the constants equal to two). Therefore, the finiteness of $C_{ns}(k,c_1)$ implies the boundedness of $\N(k,c_1)$.
\end{proof}

\vspace{0.8cm}

{\small
\noindent
\textsc{BALLICO Edoardo} \\
Dipartimento di Matematica,
Universit\`a di Trento, \\
38123 Povo (TN), Italy \\
e-mail: \texttt{ballico@science.unitn.it}
\\[4pt]
\textsc{MALASPINA Francesco} \\
Dipartimento di Matematica,
Politecnico di Torino, \\
Corso Duca degli Abruzzi 24,
10129 Torino, Italy \\
e-mail: \texttt{francesco.malaspina@polito.it}
\\[4pt]
\textsc{VALABREGA Paolo} \\
Dipartimento di Matematica,
Politecnico di Torino, \\
Corso Duca degli Abruzzi 24,
10129 Torino, Italy \\
e-mail: \texttt{paolo.valabrega@polito.it}
\\[4pt]
\textsc{VALENZANO Mario} \\
Dipartimento di Matematica,
Universit\`a di Torino, \\
via Carlo Alberto 10,
10123 Torino, Italy \\
e-mail: \texttt{mario.valenzano@unito.it}
}

\end{document}